\documentclass[twoside,a4paper,11pt]{article}
\usepackage[top=4cm, bottom=4cm, left=3cm, right=3cm]{geometry}

\usepackage{macros}

\renewcommand{\phi}{\varphi}
\newcommand{\cpl}{C_{\msf{PL}}}
\newcommand{\cals}{C_{\msf{bLS}}}
\newcommand{\cls}{C_{\msf{LS}}}
\renewcommand{\cap}{C_{\msf{bP}}}
\newcommand{\cp}{C_{\msf{P}}}
\newcommand{\bal}{ballistic}
\newcommand{\Bal}{Ballistic}

\usepackage{amsmath}               
  {
      \theoremstyle{plain}
      
  }

\begin{document}
\title{The ballistic limit of the log-Sobolev constant
equals the Polyak--\L{}ojasiewicz constant}
\author{Sinho Chewi\thanks{Department of Statistics
and Data Science, Yale University, \texttt{sinho.chewi@yale.edu}.}
\and
Austin J.\ Stromme\thanks{Department of Statistics, ENSAE/CREST, \texttt{austin.stromme@ensae.fr}.}
}

\maketitle

\begin{abstract}
The Polyak--\L{}ojasiewicz (PL) constant of a function $f \colon \R^d\to\R$ characterizes the best exponential rate of convergence of gradient flow for $f$, uniformly over initializations.
Meanwhile, in the theory of Markov diffusions, the log-Sobolev (LS) constant plays an analogous role, governing the exponential rate of convergence for the Langevin dynamics from arbitrary initialization in the Kullback--Leibler divergence. We establish a new connection between optimization and sampling
by showing that the low temperature limit $\lim_{t\to 0^+} t^{-1} \cls(\mu_t)$ of the LS constant of $\mu_t \propto \exp(-f/t)$ is exactly the PL constant of $f$, under mild assumptions.
In contrast, we show that the corresponding limit for the Poincar\'e constant is the inverse of the smallest eigenvalue of $\nabla^2 f$ at the minimizer.
\end{abstract}

\section{Introduction}

Let $f\colon \R^d \to \R$ denote a fixed twice continuously
differentiable function and, for each $t > 0$,
consider the probability
measure $\mu_t \deq  \frac{1}{Z_t}\, e^{-f/t}$
for an appropriate normalizing constant $Z_t$.
In this paper, we study the asymptotic behavior of the log-Sobolev
constant of $\mu_t$ in the low temperature regime $t \to 0^+$.

Recall that
for a probability measure $\nu$ on $\R^d$ such
that $\nu \ll \mu_t$ and $\bigl(\frac{\ud \nu}{\ud \mu_t}\bigr)^{1/2}$ is compactly supported and smooth, we define the 
Kullback--Leibler (KL) divergence, and Fisher information of $\nu$ relative to $\mu_t$, as
$$
\KL{\nu}{\mu_t}  \deq  \int  \log \frac{\ud \nu}{\ud \mu_t}\, \ud \nu\,, \quad \quad 
\FI{\nu}{\mu_t}  \deq  \int \bigl\|\nabla \log \frac{\D\nu}{\D\mu_t}\bigr\|^2\, \ud \nu\,,
$$ respectively.
The log-Sobolev constant $\cls(\mu_t)$ is then defined to be the smallest constant
$C > 0$ such that,
for all compactly supported probability measures $\nu$ with smooth density, %
it holds that
\begin{equation}\label{eqn:lsi_defn}
\KL{\nu}{\mu_t} \leqslant \frac{C}{2}\FI{\nu}{\mu_t}\,.
\end{equation}
The log-Sobolev constant is the fundamental quantity
which governs the exponential rate of convergence for the Langevin dynamics in KL divergence~\cite{bakry2014analysis}.

In the low temperature limit $t\to 0^+$, the behavior of the log-Sobolev constant $\cls(\mu_t)$ is intimately
related to the optimization landscape of $f$.
Indeed, the Langevin dynamics ${(X^t_s)}_{s\ge 0}$ for $\mu_t$ can be written, up to a rescaling of time, as
\begin{equation}\label{eqn:re-scaled_Langevin}
\ud X_s^t = - \nabla f(X_s^t)\, \ud s + \sqrt{2t}\, \ud B_s\,.
\end{equation}
In particular, as $t\to 0^+$, these dynamics formally converge to the gradient flow of $f$. 
Is is therefore intuitive that the convergence rate for the Langevin dynamics for $\mu_t$ should reflect the convergence rate for the gradient flow for $f$.
For example, if $f$ has multiple local
minima, then the Langevin dynamics converge
exponentially slowly, and the classical Eyring--Kramers
formula gives precise asymptotics for this blow-up~\cite{eyring1935activated,kramers1940brownian},
which were rigorously established
in~\cite{bovier2004metastability}. These results
were extended to characterize the exponential blow-up of the
log-Sobolev constant $\cls(\mu_t)$ in~\cite{menz2014poincare}.
In the case where $f$ has a ``benign landscape",
meaning that it has no spurious local minima and constant
curvature around its critical points,
it is known in various settings that $\cls(\mu_t)$ remains of constant
order~\cite{menz2014poincare,kinoshita2022improved, li2023riemannian}.

But in contrast to these cases, if $f$ has an ideal optimization
landscape, namely $f$ is $\alpha$-strongly convex, then $\cls(\mu_t) \leqslant t/\alpha$, as a consequence of the Bakry{--}\'Emery theory~\cite{bakry2014analysis}.
In this paper, we develop an understanding of when the log-Sobolev constant is expected to exhibit this scaling.
Since the vanishing noise limit is also referred to as ``ballistic'', we refer to the following quantity as
the {\it \bal{} log-Sobolev constant of $f$},
which we define as
\begin{equation}\label{eqn:cals}
\cals(f)  \deq  \lim_{t\to 0^+} \frac{\cls(\mu_t)}{t}\,,
\end{equation}
provided it exists.

\paragraph*{Main result.}
Our main result
gives an exact characterization of $\cals(f)$ in terms
of a fundamental constant from optimization,
namely the {\it Polyak--\L{}ojasiewicz (PL) constant of $f$}~\cite{lojasiewicz1963topological, polyak1964gradient}.
Assume that $f$ is bounded below, 
so that $f_\star  \deq  \inf_{x\in\R^d} f(x) >-\infty$.
Then the PL constant $\cpl(f)$ is
defined to be the least constant $C > 0$ such that
\begin{equation}\label{eqn:PL_defn}
f(x) - f_\star \leqslant \frac{C}{2}\,\|\nabla f(x)\|^2\,, \quad \quad \forall\, x \in \R^d\,,
\end{equation} or $+\infty$ if no such constant exists.
As we recall in more detail in Section~\ref{sec:background},
the PL constant can be
equivalently characterized as the best possible rate of exponential
convergence of gradient flow, uniformly over initializations,
and is thus a fundamental constant associated to $f$.
If $f$ is $\alpha$-strongly convex, then it is well-known that $\cpl(f) \leqslant 1/\alpha$,
yet $\cpl(f)$ can be finite even when $f$ is non-convex~\cite{karimi2016linear}.
Although the PL constant is therefore strictly weaker than
strong convexity, it
yields comparable optimization guarantees~\cite{karimi2016linear},
and thus forms the cornerstone of modern non-convex optimization.

\begin{theorem}[Main result: \bal{} log-Sobolev equals PL]\label{thm:main}
Suppose that $f \in C^2(\R^d)$
has a unique global minimizer,
and that there is a constant $L> 0$
so that $\Delta f \leqslant L\,(1+\|\nabla f\|^2)$.
For each $t > 0$, let $\mu_t  \deq  \frac{1}{Z_t}\, e^{-f/t}$ be a probability measure.
Then $\cals(f)$, defined in~\eqref{eqn:cals},
exists if and only if $\cpl(f) < \infty$,
and in this case
$$
\cals(f) = \cpl(f)\,.
$$
\end{theorem}
The proof of this result is broken up into a lower bound in Section~\ref{sec:lower_bound}
and an upper bound in Section~\ref{sec:upper_bound}.
Although the result is presented for the asymptotic quantity
$\cals(f)$, we also extract non-asymptotic bounds for $\cls(\mu_t)$
in Remark~\ref{rmk:quantiative_log_sobolev}.
We remark here that the unique minimizer assumption is necessary for a connection
 between $\cpl(f)$ and $\cals(f)$;
see below for more discussion on this point.

En route to proving Theorem~\ref{thm:main}, we also establish the
following characterization of the {\it \bal{} Poincar\'e constant of $f$}, defined analogously as
\begin{equation}\label{eqn:defn_cap}
\cap(f)  \deq  \lim_{t\to 0^+}\frac{\cp(\mu_t)}{t}\,.
\end{equation} (We recall the definition of the Poincar\'e constant $\cp(\cdot)$ in Section~\ref{sec:background}.)

\begin{theorem}[\Bal{} Poincar\'e constant]\label{thm:main_poincare}
    Suppose that $f \in C^2(\R^d)$
    has a unique global minimizer $x_\star$, $\cpl(f) < \infty$, and
    there exists $L> 0$
    so that $\Delta f \leqslant L\,(1+\|\nabla f\|^2)$.
    Then $\cap(f)$, defined in~\eqref{eqn:defn_cap}, exists and
    \begin{align*}
        \cap(f)
        &= \frac{1}{\lambda_{\min}(\nabla^2 f(x_\star))}\,.
    \end{align*}
\end{theorem}
This result is proved in Section~\ref{subsec:poincare_proof}.
Note that Theorem~\ref{thm:main_poincare}
is well-posed, since the
Hessian at the unique global minimizer is positive definite 
when $\cpl(f) < \infty$ (we recall
background on PL functions in Section~\ref{sec:background}).
And, as for Theorem~\ref{thm:main},
although the result is presented for the
asymptotic quantity $\cap(f)$,
we also extract non-asymptotic
bounds from the proof in Remark~\ref{rmk:poincare_quantitative}.
Theorem~\ref{thm:main_poincare}
reveals that under our assumptions, $\cals(f)$ depends on the \emph{global} optimization landscape of $f$, whereas $\cap(f)$ only depends on the \emph{local} behavior of $f$ around its minimizer.
On the other hand, it is known that $\limsup_{t\to 0^+}\frac{1}{t}\,\cp(\mu_t) < \infty$
when $f$ merely has a benign landscape, including,
in particular, some $f$ such that $\cpl(f) = \infty$~\cite{Li21Blog, kinoshita2022improved, li2023riemannian}.
Theorem~\ref{thm:main_poincare} therefore leaves open
a precise characterization of $\cap(f)$ in the
setting where $\cpl(f) = \infty$ yet its landscape is benign;
 we leave this interesting direction
to future work.

The remainder of the paper is organized as follows.
In the remainder of this section we give some remarks,
review additional related works, and finally fix our basic notation.
In Section~\ref{sec:background},
we collect together relevant background material and prove Theorem~\ref{thm:main_poincare}.
Then, we establish the lower bound for Theorem~\ref{thm:main} in Section~\ref{sec:lower_bound},
and the upper bound in Section~\ref{sec:upper_bound}.

\paragraph*{Interpretation via Otto calculus.}
Theorem~\ref{thm:main} is appealing in light of the interpretation, due to Otto~\cite{otto2001porousmedium}, of a ``Riemannian geometry'' over the space $\mc P_2(\R^d)$ of probability measures with finite second moment.
In this interpretation, the distance between two measures is given by the $2$-Wasserstein distance; the KL divergence $\KL{\cdot}{\mu_t}$ is viewed as a functional over this space; the Fisher information $\FI{\cdot}{\mu_t}$ is the squared norm of the gradient of $\KL{\cdot}{\mu_t}$; and the marginal law of the Langevin diffusion~\eqref{eqn:re-scaled_Langevin} is the gradient flow of $\KL{\cdot}{\mu_t}$.
The seminal work~\cite{otto2000generalization} was the first to introduce this connection.

From this perspective, the LSI arises precisely as a PL inequality for the KL divergence.
Indeed, just as the PL inequality governs the exponential rate of convergence for the gradient flow of $f$, the LSI governs the exponential rate of convergence in KL divergence of the marginal law of the Langevin diffusion.

When we write the density of the stationary distribution as $\mu\propto e^{-f}$, properties of the negative log-density $f$ give rise to properties of $\mu$.
Namely, it is known that strong convexity of $f$ implies strong convexity of the functional $\KL{\cdot}{\mu}$ over $\mc P_2(\R^d)$.
It is therefore quite natural to ask whether a PL inequality for $f$ implies a PL inequality for $\KL{\cdot}{\mu}$, i.e., a log-Sobolev inequality for $\mu$.
Although natural, it appears that no such connection was known prior to our work, as well as the recent concurrent work of~\cite{chen2024optimization},
which we discuss in detail below.
Theorem~\ref{thm:main} establishes such a correspondence in the low temperature limit.

\paragraph*{Beyond the unique minimizer case.}
We remark on the assumption that $f$ admits a
unique global minimizer.
In fact, this assumption is necessary
for a connection between the \bal{} log-Sobolev
constant and $\cpl(f)$.
Indeed, consider the prototypical example
where $f(x) = \frac{\alpha}{2}\,d^2(x, K)$ for a convex set $K \subset \R^d$ with non-empty interior, where $d^2(\cdot,K)$ denotes the squared distance function to $K$.
Then $\cpl(f) \leqslant 1/\alpha$, yet
$e^{-f/t}$ may not even be integrable.
Even when $\mu_t$ is well-defined, the log-Sobolev
constants $\cls(\mu_t)$ will not converge to $0$,
since $\mu_t$ itself will converge to the uniform measure
on $K$, so that $\cals(f) = \infty$.

On the other hand, the fact that
$\cpl(f)$ can be finite even when $f$ has multiple global minima is indeed one of the appealing properties of the PL condition, particularly
when compared with strong convexity.
One may thus hope that, so long as $\mu_t$ is well-defined,
the limiting behavior 
of $\cls(\mu_t)$ still contains information about the optimization
properties of $f$.
Thus, we ask whether the following more general limit holds
as soon as $\cls(\mu_0)  \deq  \lim_{t\to 0^+}\cls(\mu_t) < \infty$:
\begin{equation}\label{eqn:conjecture}
\lim_{t\to 0^+}\frac{\cls(\mu_t) - \cls(\mu_0)}{t} = \cpl(f)\,.
\end{equation}
Theorem~\ref{thm:main} corresponds to the case $\cls(\mu_0) = 0$ in~\eqref{eqn:conjecture}.
A proof of~\eqref{eqn:conjecture}
with exact constant appears out of reach of our proof techniques, since
we strongly rely
on the strong convexity of $f$ around its unique minimizer.

\paragraph*{Related work.}
Motivated by applications in statistical physics,
the study of the convergence of Langevin dynamics in the low
temperature regime has historically
focused on the
setting where the potential $f$ has multiple local minima.
Here, the convergence is exponentially slow, and the Arrhenius
law predicts that the rate is proportional
to the height of the energy barrier between the local minima~\cite{holley1989asymptotics,Ber13Kramer}.
The Eyring--Kramers formula refines the Arrhenius law to include
a precise description of the prefactor in front of the exponential in terms of the spectra of the Hessians at the local minimizer and the barrier~\cite{eyring1935activated,kramers1940brownian}.
The Eyring--Kramers formula was not rigorously
proved until the papers~\cite{bovier2004metastability,gayrard2005metastability},
which additionally provided asymptotics
for the spectrum of the weighted Laplacian,
including the Poincaré constant.
These results were extended to the log-Sobolev
constant in~\cite{menz2014poincare}, building on Lyapunov
criteria developed in~\cite{bakry2008simple,cattiaux2010note}.
We refer to the survey~\cite{Ber13Kramer} for more
discussion of the regime in which $f$ has multiple local
minima.

When the Langevin diffusion with small temperature
is used for finding minima of the function $f$, it is generally
called ``annealing": intuitively, the
addition of Brownian motion allows the particle to escape 
bad regions of the landscape.
Annealing is a classical and natural approach~\cite{gelfand1991recursive},
which has seen a resurgence of interest in recent years~\cite{dalalyan2017further, raginsky2017non,xu2017global,zhang2017hitting,tzen2018local},
largely motivated by major recent progress in sampling~\cite{chewi2023log}.
Of course, by the Arrhenius law, annealing
converges exponentially slowly when $f$ has multiple local minima. However, one can also consider annealing
when $f$ has only one local minimizer (the global minimizer).
Here, it turns out that the Langevin diffusion---as well as the
relevant functional inequality---behaves completely differently than in the case
of multiple local minimizers. In particular,
so long as $f$ has constant curvature around each
spurious critical point,
the Poincaré constants are $O(t)$, the log-Sobolev
constants are $O(1)$~\cite{menz2014poincare,li2023riemannian},
and annealing yields a polynomial-time algorithm for optimization~\cite{Li21Blog, kinoshita2022improved,li2023riemannian}.
In this paper, we essentially consider the regime where $f$ not 
only has just one local minimizer, but also has just one
critical point;
as far as we are aware, no other work has considered
the asymptotic behavior of the functional inequalities in
this regime.

The PL condition was first considered, independently,
by~\cite{polyak1964gradient} and in greater generality 
by~\cite{lojasiewicz1963topological}.
It was Polyak~\cite{polyak1964gradient} who first showed
it implies linear convergence of gradient descent.
The work~\cite{lojasiewicz1963topological} was focused on applications
in real algebraic geometry; we refer to~\cite{colding2014lojasiewicz}
for an overview of some modern developments
in this vein.
While the PL condition is quite natural for optimization, it appears to have
only recently received sustained
interest~\cite{karimi2016linear},
motivated by the widespread empirical success of first-order methods
for non-convex objectives.
Recent work, in particular, has been focused on
applications of the PL condition to the theory of deep learning, where it is known
to hold for sufficiently wide networks near initialization~\cite{liu2022loss}.
We have also found the PL condition to be strikingly powerful in our past work,
both in non-convex Riemannian optimization~\cite{chewi2020gradient,altschuler2021averaging},
as well as statistics~\cite{rigollet2022sample,stromme2023minimum}, leading us to believe that the PL inequality merits further study.

To the best of our knowledge, there is only
one other work which considers
functional inequalities under a PL assumption for the 
negative log-density: the concurrent paper~\cite{chen2024optimization},
which uses Lyapunov functions
and is focused on non-asymptotic bounds.
In particular, they consider $f$ with potentially more than one
global minimizer:
their bounds for $\cls(\mu_t)$ are of order $t^{-1}$ when
$f$ admits multiple minimizers, and of constant
order when $f$ does have a unique minimizer.
By contrast, our main Theorem~\ref{thm:main}
provides a new characterization of the PL condition
itself by
 showing $\cls(\mu_t) = t\,\cpl(f) + o(t)$ when $f$ has a unique minimizer.
In particular, our non-asymptotic upper bounds
on $\cls(\mu_t)$
are better by a factor of $t$, but
also capture the exact
leading order constant.
Our analysis for $\cls(\mu_t)$ also largely differs from that of~\cite{chen2024optimization},
in that we mainly rely on a novel direct argument to establish the LSI, rather than Lyapunov criteria.
While both papers use Lyapunov functions to control $\cp(\mu_t)$,
theirs uses a different Lyapunov function and more customized analysis, which, in particular, incurs
additional $f$-dependent constants
in the unique minimizer case, as compared to ours.
In particular,
the exact asymptotics for $\cp(\mu_t)$ from Theorem~\ref{thm:main_poincare} cannot be recovered from their bounds.

We finally mention that in the low temperature regime,~\cite{Che+23FisherLower} showed that the problem of finding a measure with Fisher information $O(d/t)$ relative to $\mu_t$ is equivalent to finding an $O(\sqrt{dt})$-stationary point of $f$, and subsequently used this to prove query lower bounds for the task of finding measures $\nu$ with small $\FI{\nu}{\mu_t}$.
This is another example of the connection between optimization and the low temperature limit of the Fisher information, in a similar spirit as investigated here.

\paragraph*{Acknowledgements.} We thank Pierre Monmarch\'e for helpful discussions, and Andre Wibisono for his involvement at an earlier stage of this project.

\paragraph*{Notation.}
Since we only work with absolutely continuous measures in this work, we abuse notation by identifying a probability measure $\nu$ with its Lebesgue density.
We write $\mc P(\R^d)$ for the space of probability measures over $\R^d$, $\mc N(m,\Sigma)$ for the Gaussian measure with mean $m$ and covariance $\Sigma$, and $T_\# \mu$ to denote the pushforward of a measure $\mu$ under the mapping $T$.
We also define $\|x\|_\Sigma^2 \deq \langle x,\Sigma\,x\rangle$.

The notation $B(x,r)$ refers to the open ball centered at $x$ with radius $r$.

\section{Preliminaries and the \bal{} Poincar\'e constant} \label{sec:background}

In this section, we collect together background material on functional
inequalities and PL functions, and then prove Theorem~\ref{thm:main_poincare}
on the \bal{} Poincar\'e constant.

\subsection{Background on functional inequalities
and PL functions}

\paragraph*{Background on functional inequalities.}
Given a probability measure $\mu \in \mathcal{P}(\R^d)$,
recall that its {\it Poincar\'e constant}, written $\cp(\mu)$,
is the least constant $C$ such that
for all smooth, compactly supported $h \colon \R^d \to \R$,
$$
\int h^2\, \ud \mu - \Bigl( \int h\, \ud \mu \Bigr)^2
\leqslant C \int \|\nabla h\|^2\, \ud \mu\,.
$$
For a function $f \colon \R^d\to\R$, for each $t > 0$ we define the probability measure $\mu_t \deq \frac{1}{Z_t} e^{-f/t}$, as before.

For $C, D \geqslant 0$,
we say that $\mu$ satisfies
a {\it $(C, D)$-defective log-Sobolev inequality (LSI)}
if, for all probability distributions $\nu$ such that $\bigl(\frac{\ud \nu}{\ud \mu} \bigr)^{1/2}$ is smooth,
$$
\KL{\nu}{\mu} \leqslant \frac{C}{2} \FI{\nu}{\mu} + D\,.
$$
Defective log-Sobolev inequalities are useful because,
in the presence of a Poincaré inequality,
they imply a full log-Sobolev inequality via the
so-called {\it tightening} procedure.
Although this is typically established via the Rothaus lemma~\cite{Rot1985},
our proofs crucially make use of the following recent improvement, from~\cite[Prop.\ 5]{wang2024uniform}.

\begin{lemma}[Improved tightening~\cite{wang2024uniform}]\label{lem:improved_tightening}
Suppose that $\mu \in \mathcal{P}(\R^d)$, $\cp(\mu) < \infty$
and $C, D \geqslant 0$.
If $\mu$ satisfies a $(C, D)$-defective
LSI,
then
$$
\cls(\mu) \leqslant 
C + \frac{D}{2} \, \cp(\mu)\,.
$$
\end{lemma}

We remark that the standard tightening
result obtained from the Rothaus lemma, e.g.,~\cite[Prop.\ 5.1.3]{bakry2014analysis}, reads
$$
\cls(\mu) \leqslant C + \bigl(\frac{D}{2} + 1 \bigr) \, \cp(\mu)\,.
$$
Hence Lemma~\ref{lem:improved_tightening} improves upon standard
tightening by a factor of $\cp(\mu)$.
In particular, the improved tightening
result has the appealing property of being
stable under taking $C = \cls(\mu)$ and $D = 0$.
This stability is key for our upper bound on
$\frac1t\, \cls(\mu_t)$, since it permits us to obtain the correct constants 
by proving a defective LSI with $D = o(1)$;
without the improved tightening our upper bound would be off by a factor of $2$.

We recall that by linearization, a log-Sobolev inequality implies a Poincar\'e inequality with the same constant,
$\cp(\mu) \leqslant \cls(\mu)$.
From the Bakry{--}\'Emery theory, if $\mu$ is $\alpha$-strongly log-concave, then $\cls(\mu) \le 1/\alpha$~\cite{bakry2014analysis}.

Finally, we shall write the Lebesgue density of a probability measure $\nu$ as $e^{-g}$.
Let us briefly comment on this point.

\begin{remark}[Notation for the negative log-density]\label{rem:negative_log-density_notation}
Suppose that $\mu$ is a probability
measure which admits a Lebesgue density, also denoted $\mu$, such that $\log \mu \in C^2(\R^d)$,
and that $\nu$ is another probability measure such that 
$\bigl( \frac{\ud \nu}{\ud \mu} \bigr)^{1/2}$ is well-defined
and smooth. We shall often write the density of
$\nu$ as $e^{-g}$ and then manipulate expressions such as $\int \|\nabla g\|^2\, e^{-g}$
 and $\int \langle \nabla g, \nabla \log \mu \rangle\, e^{-g}$. These expressions should be understood
 as $4\int \|\nabla e^{-g/2} \|^2$ and $-\int \langle \nabla e^{-g},
 \nabla \log \mu \rangle$, respectively.
 These latter expressions are well-defined because we can write $e^{-g} =\frac{\ud \nu}{\ud \mu} \, \frac{\ud \mu}{\ud \mc L^d}$,
 so that $e^{-g/2}$, as well as $e^{-g}$, is $C^2$.
\end{remark}

\paragraph*{Background on PL functions.}
We begin by recalling the dynamical
formulation of the PL constant.
Although the static formulation of the PL constant~\eqref{eqn:PL_defn}
is ultimately more convenient for our proofs, the dynamical
formulation is more intuitive, and emphasizes
the fundamental nature of this constant.

Suppose $h\in C^2(\R^d)$ with
a finite infimum $h_\star \deq \inf_{x \in \R^d} h(x) > -\infty$.
Consider the gradient flow for $h$ initialized
at any $x_0 \in \R^d$, so that $\dot x_s = -\nabla h(x_s)$.
Then the dynamic Polyak--\L{}ojasiewicz constant
of $h$, written $\cpl^{\msf{dyn}}(h)$, is defined
to be the least constant $C$ such that,
for all $s \in [0, \infty)$, we have
\begin{equation}\label{eqn:exp_decrease_PL}
h(x_s) - h_\star \leqslant e^{-2s/C}\,(h(x_0) - h_\star)\,,
\end{equation} uniformly over all initial conditions $x_0 \in \R^d$,
or $+\infty$ if no such constant exists.
The next proposition shows that this is precisely $\cpl(h)$.
\begin{prop}[Dynamic and static PL constants
are equal]
Suppose $h \in C^2(\R^d)$ with a finite infimum
$h_\star \deq \inf_{x \in \R^d} h(x) > -\infty$.
Then $\cpl(h) < \infty$ if and only if $\cpl^{\msf{dyn}}(h) < \infty$,
and in this case
$$
\cpl^{\msf{dyn}}(h) =\cpl(h)\,.
$$
\end{prop}
\begin{proof}
Suppose $\cpl(h) < \infty$. Then
$$
\partial_s(h(x_s) - h_\star) = -\|\nabla h(x_s)\|^2 \leqslant 
-\frac{2}{\cpl(h)}\,(h(x_s) - h_\star)\,.
$$ Grönwall's inequality thus implies that~\eqref{eqn:exp_decrease_PL}
holds with $C= \cpl(h)$, so that $\cpl^{\msf{dyn}}(h)
\leqslant \cpl(h)$.

On the other hand, suppose $\cpl^{\msf{dyn}}(h) < C < \infty$.
Then, fix any $x_0 \in \R^d$, and note that
$$
-\|\nabla h(x_0)\|^2 = \lim_{s \to 0^+} \frac{h(x_s) - h(x_0)}{s}
\leqslant (h(x_0) - h_\star) \, \lim_{s\to 0^+} \frac{e^{-2s/C} - 1}{s}
= -\frac{2}{C} \, (h(x_0) - h_\star)\,.
$$
Therefore $\cpl(h) \leqslant \cpl^{\msf{dyn}}(h)$.
The result follows.
\end{proof}

For our proofs, we frequently use two important consequences
of the PL condition.
The following is known as the {\it quadratic growth
inequality}.
The earliest proof we are aware
of is a beautiful gradient
flow argument in the optimal
transport literature, due to~\cite{otto2000generalization}.
A closely related proof appears in~\cite{karimi2016linear}.

\begin{prop}[PL functions have quadratic growth]
Suppose $h \in C^2(\R^d)$
with $h_\star  \deq  \inf_{x \in \R^d} h(x) > -\infty$,
and such that $\cpl(h) < \infty$.
Let $S$ denote the set of minimizers of $f$.
Then, for all $x\in\R^d$,
\begin{equation}\label{eqn:qg}
   \frac{1}{2\cpl(h)}\,d(x, S)^2 \leqslant h(x) - h_\star\,,
\end{equation}
where $d(x, S)  \deq  \inf_{y \in S} \|x - y\|$ is defined
as usual.
\end{prop}

We remark that in the literature on optimal
transport, the above statement is the well-known fact that the log-Sobolev inequality
implies Talagrand's transport--entropy inequality~\cite{otto2000generalization}.
When $h$ has a unique minimizer,
the quadratic growth inequality~\eqref{eqn:qg}
implies the following useful control on the Hessian.

\begin{prop}[Hessian of PL functions at the minimizer]
\label{prop:PL_implies_Hessian_control}
Suppose $h \colon \R^d \to \R$ is such that $\cpl(h) < \infty$,
and that $h$ has a unique minimizer $x_\star$, around which it is
$C^2$. Then the quadratic growth inequality~\eqref{eqn:qg}
implies that
\begin{equation}\label{eqn:Hessian_at_minimizer}
\nabla^2 h (x_\star ) \succeq \frac{1}{\cpl(h)}\,I\,.
\end{equation}
\end{prop}

\subsection{\Bal{} Poincar\'e constant
of PL functions}
\label{subsec:poincare_proof}

In this subsection, we study the Poincar\'e constants $\cp(\mu_t)$,
as well as the \bal{} Poincar\'e constant defined in~\eqref{eqn:defn_cap}.
We perform this study for several reasons.
First,
in the course of establishing an upper bound for $\cals(f)$,
we will require the preliminary estimate $\cp(\mu_t) = O(t)$ in order to apply tightening.
While such bounds have been found under
weaker assumptions than PL (allowing for saddle points
and local maxima) in previous works~\cite{menz2014poincare, Li21Blog, kinoshita2022improved, li2023riemannian},
in the case of PL functions we can establish
improved non-asymptotic bounds.
Finally, these non-asymptotic bounds
permit us to derive the exact characterization of the
\bal{} Poincaré constant from Theorem~\ref{thm:main_poincare}.

\begin{lemma}[Upper bound on the \bal{} Poincar\'e constant for 
PL functions]\label{lem:improved_asymptotic_poincaré_under_PL}
Suppose that $f \in C^2(\R^d)$, has a unique minimizer $x_\star$,
is such that $\cpl \deq \cpl(f) < \infty$,
and that $\Delta f \leqslant L_0 + L_1\,\|\nabla f\|^2$ for some constants $L_0, L_1 > 0$.
Let $\alpha, r_0 > 0$ be such that $\nabla^2 f(x) \succeq \alpha I$ for all $x \in B(x_\star, r_0)$.
Define $\delta  \deq  r_0^2/(2\cpl)$,
take any
$k \geqslant 1$, and put
$$
t_0  \deq  \frac{2\delta}{\cpl L_0 + 2\delta L_1 + 2k-2}\,.
$$
Then for all $t < t_0$,
we have
\begin{equation}\label{eqn:non_asymptotic_poincare_constant_bound}
\cp(\mu_t) 
\leqslant \frac{\cpl t/ k}{1 - L_1 t - (k - 1)t/\delta - \cpl L_0 t/(2\delta)}
+ 
\Bigl(1 +\frac{\cpl L_0 t}{\delta - \delta L_1 t - (k - 1)t - \cpl L_0 t/2}
\Bigr)\, \frac{t}{\alpha}\,.
\end{equation}
In particular,
\begin{equation}\label{eqn:asymptotic_poincare_constant_bound}
    \limsup_{t\to 0^+}\frac{\cp(\mu_t)}{t} \leqslant \frac{1}{\lambda_{\min}(\nabla^2f(x_\star))}\,.
\end{equation}
\end{lemma} 

We emphasize that the statement is not vacuous since,
under the hypotheses,
Proposition~\ref{prop:PL_implies_Hessian_control} 
implies
there is always some radius for which
$\nabla^2f$ is positive-definite on $B(x_\star,r_0)$.
\begin{remark}\label{rmk:poincare_quantitative}
    Consider the smooth case $\nabla^2 f \preceq \beta I$, and take $L_0 = \beta d$, $L_1 = 0$.
    To obtain a non-asymptotic bound which is loose but more interpretable, we can take $k=1$ and $\alpha = 1/(2\cpl)$.
    Then, for all $t \le \delta/(\cpl \beta d)$,~\eqref{eqn:non_asymptotic_poincare_constant_bound} yields $\cp(\mu_t) \le 5\cpl t$.
\end{remark}

Before we turn to the proof,
let us apply this result
to prove Theorem~\ref{thm:main_poincare}.

\begin{proof}[Proof of Theorem~\ref{thm:main_poincare}]
    Without loss of generality, we may suppose that $f(x_\star) = 0$.
    For any unit vector $v \in \R^d$, by applying the Poincar\'e inequality to the test function $h = \langle v, \cdot\rangle$, we obtain $\cp(\mu_t) \ge \int \langle v, x-x_\star\rangle^2 \, \ud \mu_t(x) - (\int \langle v,x-x_\star\rangle\,\ud \mu_t(x))^2$.
    Further, by applying the change of variables $x = x_\star + \sqrt t\,z$,
    \begin{align}
        \cp(\mu_t)
        &\ge \int \langle v,x - x_\star\rangle^2\,\frac{e^{-f(x)/t}}{Z_t}\,\ud x - \Bigl(\int \langle v, x-x_\star\rangle\,\frac{e^{-f(x)/t}}{Z_t}\,\ud x\Bigr)^2 \nonumber\\
        &= t\,\Bigl[\int \langle v,z\rangle^2\,a_t b_t(z)\,\ud z - \Bigl(\int\langle v,z\rangle\,a_t b_t(z)\,\ud z\Bigr)^2\Bigr]\,,\label{eq:bal_poincare_lb}
    \end{align}
    where
    \begin{align*}
        a_t \deq \frac{\sqrt{\det(2\pi t\,[\nabla^2 f(x_\star)]^{-1})}}{Z_t}\,, \qquad b_t(z) = \frac{e^{-f(x_\star + \sqrt t\,z)/t}}{\sqrt{\det(2\pi \,[\nabla^2 f(x_\star)]^{-1})}}\,.
    \end{align*}
    We observe that $\int a_t b_t(z)\,\ud z = 1$, $b_t(z)\,\ud z \to \bar\mu \deq \mc N(0, [\nabla^2 f(x_\star)]^{-1})$ pointwise, and hence $\int b_t(z)\,\ud z \to 1$ by dominated convergence.
    Then,
    \begin{align*}
        1
        = \liminf_{t\to 0^+} \frac{1}{\int b_t}
        \le \liminf_{t\to 0^+} a_t
        \le \limsup_{t\to 0^+} a_t
        \le \limsup_{t\to 0^+} \frac{1}{\int b_t} = 1\,.
    \end{align*}
    We conclude that $a_t \to 1$, hence $a_t b_t \to \bar\mu$ pointwise, hence $a_t b_t \to \bar \mu$ in $L^1(\R^d)$ by Scheff\'e's lemma.
    Moreover, for any $R > 0$, the quadratic growth inequality~\eqref{eqn:qg} shows that
    \begin{align*}
        \int_{\|z\|\ge R} \|z\|^2 \,a_t b_t(z)\,\ud z
        \le \frac{a_t}{\sqrt{\det(2\pi \,[\nabla^2 f(x_\star)]^{-1})}} \int_{\|z\|\ge R} \|z\|^2 \exp\bigl(-\frac{\|z\|^2}{2\cpl(f)}\bigr)\,\ud z
    \end{align*}
    and therefore $\lim_{R\to\infty} \limsup_{t\to 0^+} \int_{\|z\|\ge R} \|z\|^2\,a_t b_t(z)\,\ud z = 0$.
    By~\cite[Thm.\ 7.12]{villani2003topics}, we have the convergence $W_2(a_t b_t, \bar \mu) \to 0$.
    From~\eqref{eq:bal_poincare_lb},
    \begin{align*}
        \liminf_{t\to 0^+} \frac{\cp(\mu_t)}{t}
        &\ge \sup_{v\in\R^d, \, \|v\| = 1} \Bigl[\int \langle v,z\rangle^2\,\ud \bar\mu(z) - \Bigl(\int \langle v,z\rangle\,\ud \bar \mu(z)\Bigr)^2\Bigr] \\
        &= \sup_{v\in\R^d, \, \|v\| = 1}\bigl\langle v, [\nabla^2 f(x_\star)]^{-1}\,v\bigr\rangle
        = \frac{1}{\lambda_{\min}(\nabla^2 f(x_\star))}\,. \qedhere
    \end{align*}
\end{proof}

To prove Lemma~\ref{lem:improved_asymptotic_poincaré_under_PL},
we use the Lyapunov criterion for the Poincar\'e inequality~\cite[Theorem 4.6.2]{bakry2014analysis},
which we include below for the reader's convenience.

\begin{lemma}[Poincar\'e Lyapunov criterion]\label{lem:sufficient}
    Suppose that there exists a Lyapunov function $W \colon \R^d \to \R$ with
$W \geqslant 1$, and a set $K \subseteq \R^d$, such that the following hold.
\begin{enumerate}
    \item[1.] The restricted measure $\mu|_K$ satisfies a Poincar\'e inequality with constant
    $C_K$, so that $\cp(\mu|_K) \leqslant C_K$.
    \item[2.] There are constants $b, \lambda > 0$ such that
    for almost every %
    $x \in \R^d$:
    $$
    -\frac{\mathcal{L}_\mu W(x)}{W(x)} \geqslant \lambda\,(1 - b\mathbbold{1}_K(x))\,,
    $$ where $\mathcal{L}_\mu$ is the generator for $\mu$,
    namely $\mathcal{L}_\mu  \deq  \Delta + \langle \nabla \log \mu, \nabla \cdot \rangle$.
\end{enumerate}
Then,
$$
\cp(\mu) \leqslant \lambda^{-1} + bC_K\,.
$$
\end{lemma}

\begin{proof}[Proof of Lemma~\ref{lem:improved_asymptotic_poincaré_under_PL}]
Assume without loss of generality that the unique minimizer of $f$
is $0$, and that $f(0) = 0$.
Define $K  \deq  B(0, r_0)$ and set $W  \deq  (1+f/\delta)^k$.
We apply Lemma~\ref{lem:sufficient}.
Notice that since $f$ is $\alpha$-strongly
convex on $K$ and $K$ is convex, we have
$$
\cp(\mu_t|_K) \leqslant \frac{t}{\alpha}\,,
$$
c.f.~\cite[Theorem 3.3.2]{Wang14Diffusion}.

On the other hand, we may compute
$$
-\frac{\mathcal{L}_{\mu_t}W}{W} =  
\Bigl\{\frac{k}{t} - \frac{k\,(k - 1)}{f + \delta}\Bigr\}\,\frac{\|\nabla f\|^2}{f + \delta}
-\frac{k\Delta f}{f + \delta}\,.
$$
Applying the assumption on $\Delta f$ and the fact that $f \geqslant 0$
yields
\begin{align*}
-\frac{\mathcal{L}_{\mu_t}W}{W} &\geqslant
\Bigl\{\frac{k}{t} - kL_1 - \frac{k\,(k - 1)}{f + \delta}\Bigr\}\,\frac{\|\nabla f\|^2}{f + \delta}
-\frac{kL_0}{f + \delta}  \\
&\geqslant \Bigl\{\frac{1}{t} - L_1 - \frac{k-1}{\delta}\Bigr\}\,\frac{k\,\|\nabla f\|^2}{f + \delta}
-\frac{kL_0}{f + \delta}\,.
\end{align*} For $t$ sufficiently small, we may apply the 
PL inequality to obtain
\begin{equation}\label{eqn:W_ratio_lower_bound}
-\frac{\mathcal{L}_{\mu_t}W}{W} 
\geqslant \Bigl\{\frac{1}{t} - L_1 - \frac{k-1}{\delta}\Bigr\}\,\frac{2k}{\cpl} \,\frac{f}{f+\delta}
-\frac{kL_0}{f + \delta}\,.
\end{equation}
When $x \not \in K$, the quadratic growth inequality~\eqref{eqn:qg} implies
$$
f(x) \geqslant \frac{1}{2\cpl}\, \|x\|^2 
\geqslant \frac{r_0^2}{2\cpl} = \delta\,.
$$ But notice that~\eqref{eqn:W_ratio_lower_bound}
is monotonic in $f$, so for $x \not \in K$,
$$
- \frac{\mathcal{L}_{\mu_t}W}{W} \geqslant 
\Bigl\{\frac{1}{t} - L_1 - \frac{k-1}{\delta}\Bigr\} \,\frac{k}{\cpl} - \frac{kL_0}{2\delta} \eqqcolon \lambda\,.
$$ So long as $t$ is as small as in the statement, $\lambda > 0$.
On the other hand, if $x \in K$, then by monotonicity of
$f$ in~\eqref{eqn:W_ratio_lower_bound}
again,
we obtain
$$
- \frac{\mathcal{L}_{\mu_t}W}{W} \geqslant 
-\frac{kL_0}{\delta}\,.
$$
Defining $b  \deq  1 + \frac{kL_0}{\delta \lambda}$,
we conclude that $W \geqslant 1$ and
$$
- \frac{\mathcal{L}_{\mu_t}W}{W} \geqslant \lambda\,(1 - b \mathbbold{1}_K)\,.
$$

Lemma~\ref{lem:sufficient} thus implies
\begin{align*}
\cp(\mu_t) \leqslant \frac1\lambda + b\cp(\mu_t|_K)
&\leqslant \frac{\cpl t/ k}{1 - L_1 t - (k - 1)t/\delta - \cpl L_0 t/(2\delta)} \\
&\qquad + 
\Bigl(1 +\frac{\cpl L_0 t}{\delta - \delta L_1 t - (k - 1)t - \cpl L_0 t/2}
\Bigr)\, \frac{t}{\alpha}\,.
\end{align*} This yields~\eqref{eqn:non_asymptotic_poincare_constant_bound}.
Forming $\frac1t\, \cp(\mu_t)$ and taking
$t \to 0^+$ yields
$$
\limsup_{t\to 0^+} \frac{\cp(\mu_t)}{t}
\leqslant \frac{\cpl}{k} + \frac{1}{\alpha}\,.
$$ But since the above is true for every $k \ge 1$ and $\alpha
< \lambda_{\min}(\nabla^2f(0))$, equation~\eqref{eqn:asymptotic_poincare_constant_bound} follows.
\end{proof}

\section{Lower bound for the \bal{} log-Sobolev constant}\label{sec:lower_bound}

In this section, we prove the following lower bound.
\begin{theorem}[Lower bound of \bal{} log-Sobolev
by PL]
\label{thm:lower_bound}
    Suppose $f\colon \R^d \to \R$
    is minimized by some $x_\star \in \R^d$
    around which it is Lipschitz continuous.
    Then, for any point $x \in \R^d$
    around which $f$ is continuously differentiable,
    we have
    $$
    f(x) - f(x_\star) \leqslant \frac12\, \Bigl(\liminf_{t\to 0^+}\frac{\cls(\mu_t)}{t}\Bigr) \, \|\nabla f(x)\|^2\,.
    $$
In particular, if $f$ is minimized by
some $x_\star \in \R^d$ and is continuously
differentiable everywhere,
then
 $$
 \cpl(f) \leqslant \liminf_{t\to 0^+}\frac{\cls(\mu_t)}{t}\,.
 $$
\end{theorem}

\begin{proof}
We assume without loss of generality that
$x_\star = 0$ and $f(x_\star) = 0$.
Take such an $x \in \R^d$ and let
$C > \liminf_{t\to 0^+}\frac1t\, \cls(\mu_t)$.
Then, there is some decreasing sequence $(t_k)_{k \in \N}$
of positive numbers such that $t_k \to 0$
and $\cls(\mu_{t_k}) < Ct_k$ for all $k \in \N$.
Let $\nu = e^{-g}$ be any smooth density
which is
supported in the neighborhood
of $x$ where $f$ is differentiable.
Then, it holds that
    \begin{align*}
        \KL{\nu}{\mu_{t_k}}
        &\le \frac{Ct_k}{2} \FI{\nu}{\mu_{t_k}}\,.
    \end{align*}
    We may expand to observe that
    \begin{align*}
        \frac{1}{t_k}\int f\,\ud \nu + \log Z_{t_k} - H(\nu)
        &\le \frac{Ct_k}{2}\int\bigl\lVert \nabla g - \frac{\nabla f}{t_k}\bigr\rVert^2\,\ud \nu \\
        &= \frac{Ct_k}{2}\int\Bigl[\frac{1}{t_k^2}\, \norm{\nabla f}^2 - \frac{2}{t_k} \, \langle \nabla f, \nabla g \rangle
        + \,\|\nabla g\|^2\Bigr]\,\ud \nu\,,
    \end{align*}
where
$H(\nu)  \deq  \int g e^{-g}$.
    
    On the other hand, let $L$ be the Lipschitz
    constant of $f$ near $0$.
    Then for $t$ sufficiently small we have
    \begin{align*}
        \log Z_t = \log \int\exp\bigl(-\frac{f}{t}\bigr)
        &\geqslant \log \int_{B(0, t)} \exp\bigl( - \frac{f}{t}\bigr)
        \ge - L + \log(t^d \omega_d) \,,
    \end{align*}
    where we write the volume of the unit ball in $\R^d$
    as $\omega_d$.
    In particular,
    $$
        \frac{1}{t_k}\int f\,\ud \nu  - L + \log(t_k^d \omega_d)
        - H(\nu)
        \leqslant
       C \, \int\Bigl[\frac{1}{2t_k}\, \norm{\nabla f}^2 -  \, \langle \nabla f, \nabla g \rangle
        + \frac{t_k}{2}\,\|\nabla g\|^2\Bigr]\,\ud \nu\,, 
    $$
    Multiplying by $t_k$ and re-arranging,
    we find
    \begin{align*}
        \int f\,\ud \nu
        &\le \frac{C}{2}\int\norm{\nabla f}^2\,\ud \nu - Ct_k \int \langle 
        \nabla f, \nabla g \rangle \, \ud \nu\\
        &\qquad{}
        + \bigl(L + H(\nu) - \log(t_k^d \omega_d)\bigr)\,t_k +\frac{Ct_k^2}{2} \int\|\nabla g\|^2\,\ud \nu\,.
    \end{align*}
     Taking $k \to \infty$, we obtain
     $$
     \int f\,\ud \nu \leqslant \frac{C}{2} \int\|\nabla f\|^2\,\ud \nu\,.
     $$
Since $C > \liminf_{t \to 0^+} \frac1t\, \cls(\mu_t)$
was arbitrary, it follows that
$$
\E_\nu f
        \le \frac12\, \Bigl(\liminf_{t\to 0^+}\frac{\cls(\mu_t)}{t}\Bigr) \int\norm{\nabla f}^2\,\ud \nu\,.
$$
Finally, since $\nu$ was an arbitrary smooth distribution
supported near $x$
and
$f$, $\|\nabla f\|^2$ are continuous around $x$,
we can take $\nu \to \delta_x$ to conclude.
\end{proof}

\section{Upper bound for the \bal{} log-Sobolev constant}\label{sec:upper_bound}

In this section, we prove the following upper bound.

\begin{theorem}[Upper bound of
\bal{} LSI by PL]\label{thm:upper_bound}
Suppose $f \in C^2(\R^d)$
has a unique global minimizer,
and that there are constants $L_0, L_1 > 0$
so that $\Delta f \leqslant L_0 + L_1\,\|\nabla f\|^2$.
For each $t > 0$, let $\mu_t  \deq  \frac{1}{Z_t}\, e^{-f/t}$.
Then
$$
\limsup_{t\to 0^+} \frac{\cls(\mu_t)}{t} \leqslant \cpl(f)\,.
$$
\end{theorem}

By combining this result with Theorem~\ref{thm:lower_bound}
from the last section, we immediately arrive at our main result, Theorem~\ref{thm:main}.

\begin{proof}[Proof of Theorem~\ref{thm:upper_bound}]
If $\cpl = \cpl(f) = \infty$, then there is nothing to show, so we may
assume $\cpl < \infty$.
Throughout the proof,
we assume without loss of generality that $x_\star = 0$ and $f(x_\star) = f(0) = 0$,
and let $\Sigma_t \deq  t\,[\nabla^2 f(0)]^{-1}$, which is well-defined
by~\eqref{eqn:Hessian_at_minimizer}.

Let $\nu$ be any
compactly supported probability measure such that $\bigl(\frac{\ud \nu}{\ud \mu_t}\bigr)^{1/2}$ is smooth
and write the Lebesgue density of $\nu$ as $e^{-g}$;
see Remark~\ref{rem:negative_log-density_notation} for an explanation of this notation.

We begin by expanding the KL divergence 
$$
\KL{\nu}{\mu_t} = -\int g e^{-g} + \frac1t  \int fe^{-g}
+ \log Z_t\,.
$$

{\bf Step 1: Gaussian log-Sobolev inequality.}
For this step, we start by rewriting the KL divergence
to obtain a $\KL{\nu}{\mathcal{N}(0, \Sigma_t)}$ term:
$$
\KL{\nu}{\mu_t} = \KL{\nu}{\mathcal{N}(0, \Sigma_t)}
+ \int \bigl(\frac{f}{t} - \frac12\, \|x\|^2_{\Sigma_t^{-1}}\bigr)\,e^{-g}
+ \log \frac{Z_t}{\sqrt{\det(2\pi \Sigma_t)}}\,.
$$
In Lemma~\ref{lem:normalizing_constant_control} below, we verify that the last term involving the ratio of normalizing constants is $o(1)$.

Next, we apply the Gaussian log-Sobolev
inequality with covariance $\Sigma_t$ (see the discussion after~\cite[Prop.\ 5.5.1]{bakry2014analysis})
to the first term, which yields
\begin{align*}
\KL{\nu}{\mu_t}  &=\KL{\nu}{\mathcal{N}(0, \Sigma_t)}
+ \int \bigl(\frac{f}{t} - \frac12\, \|x\|^2_{\Sigma_t^{-1}}\bigr)\,e^{-g} + o(1)\\
&\leqslant \frac12 \int 
\|\nabla g - \Sigma_t^{-1} x\|^2_{\Sigma_t}\,e^{-g} + \int \bigl(\frac{f}{t} - \frac12\, \|x\|^2_{\Sigma_t^{-1}}\bigr)\,e^{-g} + o(1) \\
&=  \frac12 \int \|\nabla g\|^2_{\Sigma_t}\, e^{-g} + \frac1t \int fe^{-g} - d + o(1)\,,
\end{align*} 
where the last equality follows by integration by parts.

{\bf Step 2: Separating into small and large scales.}
The second step is to split into a small-scale region $B  \deq  B(0, r)$
for some $r = r(t)$ to be chosen later,
and its complement $B^\comp$.
Define the modulus of continuity of the Hessian,
\begin{align*}
    \rho(r)
    &\deq \sup_{x\in B(0,r)} \norm{\nabla^2 f(x) - \nabla^2 f(0)}_{\rm op}\,,
\end{align*}
and note that since $f$ is of class $C^2$, we have $\rho(r) \to 0$ as $r\to 0^+$.
In the small-scale region $B$, we use Taylor expansion to observe that if $x \in B$, then
$$
f(x) = \frac12\,\langle x, \nabla^2 f(0)\,x\rangle + O(r^2\rho(r))\,, \qquad \quad 
\nabla f(x) = \nabla^2 f(0)\,x + O(r\rho(r))\,.
$$ Therefore, for all $x \in B$,
\begin{align*}
    \|\nabla f(x)\|^2_{\Sigma_t}
    &= t\,\langle \nabla f(x), \nabla^2 f(0)^{-1}\, \nabla f(x)\rangle
    = t\,\bigl\{\langle x, \nabla f(x)\rangle + O\bigl(r\rho(r)\,\norm{\nabla^2 f(0)^{-1}\,\nabla f(x)}\bigr)\bigr\} \\
    &= t\,\bigl\{\langle x, \nabla^2 f(0) \,x\rangle + O\bigl(r^2 \rho(r) + \cpl r^2 \rho(r)^2\bigr)\bigr\}
\end{align*}
where we used $\nabla^2 f(0)^{-1} \preceq \cpl I$ from~\eqref{eqn:Hessian_at_minimizer}.
Hence,
$$
\frac1t \int_B fe^{-g}  = \frac{1}{2t^2} \int_B \|\nabla f\|^2_{\Sigma_t}\,e^{-g}
+ \underbrace{O\Bigl( \frac{r^2\rho(r) + \cpl r^2 \rho(r)^2}{t}\Bigr)}_{\eqqcolon \msf E_0(r)}\,.
$$
On the large-scale region $B^\comp$ we use the fact that $\Sigma_t \preceq \cpl tI$ via~\eqref{eqn:Hessian_at_minimizer},
and the PL inequality directly
to note that
$$
\int_{B^\comp} \bigl\{ \frac{f}{t} + \frac12\,\|\nabla g\|_{\Sigma_t}^2 \bigr\}\,e^{-g}
\leqslant \frac{\cpl t}{2} \int_{B^\comp} \bigl\{ 
\frac{\|\nabla f\|^2}{t^2} + \|\nabla g\|^2\bigr\}\,e^{-g}\,.
$$
Combining these bounds yields
\begin{align*}
    \KL{\nu}{\mu_t} &\leqslant  \frac12 \int \|\nabla g\|^2_{\Sigma_t}\, e^{-g} + \frac1t \int f e^{-g} - d + o(1)  \\
    &= \int_B \bigl\{ \frac{f}{t} + \frac12\,\|\nabla g\|_{\Sigma_t}^2 \bigr\}\,e^{-g}
    + \int_{B^\comp} \bigl\{ \frac{f}{t} + \frac12\,\|\nabla g\|_{\Sigma_t}^2 \bigr\}\,e^{-g}
    - d + o(1) \\
    &\leqslant \frac{1}{2} \int_B \bigl\{ \frac{1}{t^2}\,\|\nabla f\|_{\Sigma_t}^2 + \|\nabla g\|_{\Sigma_t}^2 \bigr\}\,e^{-g}
    + \frac{\cpl t}{2} \int_{B^\comp} \bigl\{ \frac{\|\nabla f\|^2}{t^2} + \|\nabla g\|^2\bigr\}\,e^{-g}
    \\
    &\qquad{} - d + \msf E_0(r) + o(1)\,.
\end{align*}
Adding and subtracting inner product terms and applying
$\Sigma_t \preceq \cpl tI$ again, we arrive at
\begin{align*}
    \KL{\nu}{\mu_t} 
     &\le \frac12 \int_B \bigl\lVert\frac{\nabla f}{t}-\nabla g \bigr\rVert^2_{\Sigma_t}\,e^{-g}
     + \frac{\cpl t}{2} \int_{B^\comp} \bigl\lVert\frac{\nabla f}{t}-\nabla g \bigr\rVert^2\,e^{-g} \\
     &\qquad{} + \frac1t \int_B \langle \nabla f, \Sigma_t\,\nabla g\rangle\,e^{-g} - d
     + \cpl \int_{B^\comp} \langle \nabla f, \nabla g \rangle\,e^{-g} + \msf E_0(r) + o(1)\\
     &\leqslant \frac{\cpl t}{2} \FI{\nu}{\mu_t} \\
     &\qquad{} + \frac1t \int_B \langle \nabla f, \Sigma_t\,\nabla g\rangle\,e^{-g} - d
     + \cpl \int_{B^\comp} \langle \nabla f, \nabla g \rangle\,e^{-g} + \msf E_0(r) + o(1)\,.
\end{align*}

{\bf Step 3: Reduction to large-scale terms.}
To handle these inner product terms, we apply
integration by parts. To this end, let $\omega_d$ denote
the standard volume measure on the unit sphere $\S^{d - 1}$.
Then, by the divergence theorem,
\begin{align*}
&\frac1t \int_B \langle \nabla f, \Sigma_t\, \nabla g\rangle\,e^{-g}
= \int_B \langle \nabla f, \nabla^2f(0)^{-1} \,\nabla g \rangle\,e^{-g} \\
&\qquad = \int_B \tr(\nabla^2 f(0)^{-1}\, \nabla^2 f) \,e^{-g}
- \int \langle w, \nabla^2f(0)^{-1}\,\nabla f(rw) \rangle\, r^{d - 1} e^{-g(rw)}\, \ud \omega_d(w)\,.
\end{align*}
Also, for all $x \in B$ we have
$$
\tr(\nabla^2 f(0)^{-1}\, \nabla^2f(x)) = d + O(\cpl d\rho(r))\,.
$$ Hence,
\begin{align*}
&\frac1t \int_B \langle \nabla f,\Sigma_t\, \nabla f\rangle \,e^{-g} - d \\
&\qquad = d \int_B e^{-g}  - d
- \int \langle w, \nabla^2f(0)^{-1}\,\nabla f(rw) \rangle\, r^{d - 1} e^{-g(rw)}\, \ud \omega_d(w) + O(\cpl d\rho(r)) \\
&\qquad \le - \int \langle w, \nabla^2f(0)^{-1}\,\nabla f(rw) \rangle\, r^{d - 1} e^{-g(rw)}\, \ud \omega_d(w) + O(\cpl d\rho(r))\,.
\end{align*}
Applying Cauchy--Schwarz, equation~\eqref{eqn:Hessian_at_minimizer},
and the fact that $\nabla f$ is Lipschitz around the minimizer,
we have
\begin{align*}
    &- \int \langle w, \nabla^2f(0)^{-1}\,\nabla f(rw) \rangle\, r^{d - 1} e^{-g(rw)}\, \ud \omega_d(w) \\
    &\qquad \le \int \{-r + O(\cpl r \rho(r))\}\, r^{d - 1} e^{-g(rw)}\, \ud \omega_d(w)
    = O(\cpl r^d \rho(r)) \int e^{-g(rw)}\,\ud \omega_d(w)\,.
\end{align*}
For ease of notation, let us write this
last integral in terms of
$s(r)  \deq r^d \int e^{-g(rw)}\, \ud \omega_d(w)$.
Applying the divergence theorem to the 
large-scale inner product term,
as well as our assumption on $\Delta f$, we obtain, 
\begin{align*}
 \int_{B^\comp} \langle \nabla f, \nabla g \rangle\,e^{-g}
 &= \int_{B^\comp} \Delta f\,e^{-g} + \int \langle
 \nabla f(rw), w \rangle \, r^{d - 1}\,e^{-g(rw)}\, \ud \omega_d(w) \\
 &\leqslant L_0 \int_{B^\comp} e^{-g} + L_1 \int \|\nabla f\|^2\,e^{-g} + \lambda_{\max}\, s(r) + O(\rho(r)\,s(r))\,,
\end{align*}
where $\lambda_{\max} = \lambda_{\max}(\nabla^2 f(0))$.

To control the spherical integral $s(r)$, we
argue that there must be some radius of the same
order such that it is bounded by the
tail $\int_{B^\comp} e^{-g}$. Indeed,
fix any $r_0 > 0$, and integrate over the annulus $\{x\in\R^d \colon r_0 \leqslant
\|x\|\leqslant 2r_0\}$ to find
$$
\int_{r_0}^{2r_0}s(r)\, \ud r = \int_{r_0 \leqslant \|\cdot\|\leqslant 2r_0}
\|\cdot\|\, e^{-g} \leqslant 2r_0 \int_{ r_0 \leqslant \|\cdot\| \leqslant 2r_0}
e^{-g} \leqslant 2r_0 \int_{\|\cdot\|\ge r_0} e^{-g}\,.
$$ It follows that there exists some $r \in [r_0, 2r_0]$, depending 
on $\nu$, such
that
$$
s(r) \leqslant 2 \int_{\|\cdot\|\ge r_0} e^{-g}\,.
$$
If we choose $r$ so that the above holds, then
\begin{align*}
     \KL{\nu}{\mu_t}
     &\leqslant \frac{\cpl t}{2} \FI{\nu}{\mu_t} + \cpl L_0 \int_{B^\comp} e^{-g} + \cpl L_1 \int\|\nabla f\|^2\,e^{-g} \\
     &\qquad{}+ O\bigl((\cpl \lambda_{\max} +\cpl \rho(r))\, s(r) + \cpl d\rho(r)\bigr) + \msf E_0(r) + o(1) \\
     &\leqslant \frac{\cpl t}{2} \FI{\nu}{\mu_t} + O\bigl(\cpl\,(L_0 + \lambda_{\max} + \rho(2r_0))\bigr) \int_{\|\cdot\|\ge r_0} e^{-g} \\
     &\qquad{} + \cpl L_1 \int\|\nabla f\|^2\,e^{-g} + {\underbrace{O(\cpl d\,\rho(2r_0))}_{\eqqcolon \msf E_1}} + {\underbrace{\msf E_0(2r_0)}_{\eqqcolon \msf E_0}} + o(1)\,,
\end{align*}
where in the last step we used $r \in [r_0, 2r_0]$ and the fact that $\msf E_0(r)$ is monotonic in $r$.

{\bf Step 4: Controlling the large-scale terms.}
We control the tail term
by observing that by the quadratic growth inequality~\eqref{eqn:qg},
we have for all $x$ such that $\|x\| \geqslant r_0$,
$$
1 \leqslant \frac{2\cpl}{r_0^2}\,f(x)
\le \frac{\cpl^2}{r_0^2}\, \|\nabla f(x)\|^2\,.
$$
Hence,
$$
\int_{\|\cdot\|\ge r_0} e^{-g} \leqslant
\frac{\cpl^2}{r_0^2} \int_{\|\cdot\|\ge r_0} \|\nabla f\|^2\,e^{-g}
\leqslant \frac{\cpl^2}{r_0^2} \int \|\nabla f\|^2\,e^{-g}\,.
$$
Next, we apply an integration by parts argument, similar to that used in~\cite[Lemma 20]{Che+24LMC}.
Namely,
\begin{align*}
\frac{1}{t^2} \int \|\nabla f\|^2\,e^{-g}
&= \FI{\nu}{\mu_t} - \int \|\nabla g\|^2\,e^{-g} + \frac{2}{t}\int
\langle \nabla f, \nabla g \rangle\,e^{-g} \\
&\leqslant \FI{\nu}{\mu_t}
+ \frac{2}{t}\int
\langle \nabla f, \nabla g \rangle\,e^{-g} = \FI{\nu}{\mu_t}
+ \frac2t \int \Delta f\,e^{-g}\,.
\end{align*}
Applying our assumption on $\Delta f$, we obtain
$$
\frac2t \int \Delta f \,e^{-g} \leqslant \frac{2L_0}{t} + \frac{2L_1}{t} \int\|\nabla f\|^2\, e^{-g}\,.
$$ So long as $t < \frac{1}{2L_1}$, the above can be arranged
to yield
$$
\frac{1}{t^2} \int \|\nabla f\|^2\,e^{-g} \leqslant \frac{1}{1-2t/L_1}\,\bigl(\FI{\nu}{\mu_t} + \frac{2L_0}{t}\bigr)\,.
$$
Thus, if we set
\begin{align*}
    \msf E_2
    \deq \cpl L_1 + \frac{\cpl^3}{r_0^2}\,\bigl(L_0 + \lambda_{\max} + \rho(2r_0)\bigr)\,,
\end{align*}
then
\begin{align*}
     \KL{\nu}{\mu_t}
     &\leqslant \frac{\cpl t}{2} \FI{\nu}{\mu_t} + O(\msf E_2) \int\|\nabla f\|^2\,e^{-g} + \msf E_0 + \msf E_1 + o(1) \\
     &\leqslant \Bigl(\frac{\cpl t}{2} + \frac{O(\msf E_2 t^2)}{1-2t/L_1}\Bigr)\FI{\nu}{\mu_t} + \msf E_0 + \msf E_1 + O\Bigl(\frac{\msf E_2 L_0 t}{1-2t/L_1}\Bigr) + o(1)\,.
\end{align*}

{\bf Step 5: Conclusion.}
We now recall that
\begin{align*}
    \msf E_0 = O\Bigl( \frac{r_0^2 \,\rho(2r_0) + \cpl r_0^2\,\rho(2r_0)^2}{t}\Bigr)\,, \qquad \msf E_1 = O(\cpl d\,\rho(2r_0))\,.
\end{align*}
We choose $r_0 = A\sqrt t$ for some $A > 0$, so that for small $t$,
\begin{align*}
    \msf E_0 = O\bigl(A^2\,\rho(2A\sqrt t) + A^2\cpl \,\rho(2A\sqrt t)^2\bigr) = o_A(1)\,,
    \qquad \msf E_1 = O\bigl(\cpl d\,\rho(2A\sqrt t)\bigr) = o_A(1)\,,
\end{align*}
and
\begin{align*}
    \frac{\msf E_2 L_0 t}{1-2t/L_1}
    &= O\Bigl(\cpl L_0 L_1 t + \frac{\cpl^3 L_0}{A^2}\,\bigl(L_0 + \lambda_{\max} + \rho(2r_0)\bigr)\Bigr)
    = O\bigl(\frac{1}{A^2}\bigr)\,, \\
    \frac{\msf E_2 t^2}{1-2t/L_1}
    &= O\Bigl(\cpl L_1 t^2 + \frac{\cpl^3 t}{A^2}\,\bigl(L_0 + \lambda_{\max} + \rho(2r_0)\bigr)\Bigr)
    = O\bigl( t^2 + \frac{t}{A^2}\bigr)\,.
\end{align*}
From Lemma~\ref{lem:improved_asymptotic_poincaré_under_PL}, $\cp(\mu_t) = O(t)$. 
Using the
improved tightening result~\cite[Prop.\ 5]{wang2024uniform}, included here as Lemma~\ref{lem:improved_tightening} for the reader's convenience, we obtain
\begin{align*}
    \cls(\mu_t)
    &\le \cpl t + O\bigl(t^2 + \frac{t}{A^2}\bigr) + \Bigl( o_A(1) + O\bigl(\frac{1}{A^2}\bigr)\Bigr)\,\cp(\mu_t)\,,
\end{align*}
or
\begin{align*}
    \limsup_{t\to 0^+} \frac{\cls(\mu_t)}{t}
    &\le \cpl + O\bigl(\frac{1}{A^2}\bigr)\,.
\end{align*}
Finally, letting $A \to \infty$ yields the result.
\end{proof}

\begin{remark}[Quantitative estimate]\label{rmk:quantiative_log_sobolev}
    Assume now that $\nabla^2 f \preceq \beta I$ and that $\nabla^3 f$ is $\gamma$-Lipschitz in the operator norm.
    Keeping the leading order terms in $t$,
    \begin{align*}
        \msf E_0
        &= O(\gamma A^3 \sqrt t)\,, \qquad \msf E_1 = O(\cpl \gamma d A\sqrt t)\,, \qquad \msf E_2 = O\Bigl(\frac{\beta \cpl^3 d}{A^2 t}\Bigr)\,.
    \end{align*}
    To balance the terms, we will take $A$ of order $t^{-1/10}$, so that $\msf E_1$ and the $o(1)$ term---quantified in Remark~\ref{rmk:quantitative_normalizing}---are both negligible.
    Specifically, by applying Remark~\ref{rmk:poincare_quantitative} and taking $A \asymp \beta^{2/5} \cpl^{3/5} d^{2/5}/(\gamma^{1/5} t^{1/10})$, up to leading order in $t$,
    \begin{align*}
        \cls(\mu_t)
        \le \cpl t + O\bigl((\beta \cpl^{7/3} \gamma^{1/3} d t)^{6/5}\bigr)\,.
    \end{align*}
\end{remark}

To conclude this section, we establish a weaker, but simple non-asymptotic defective log-Sobolev inequality.
In particular, the following proposition does not require $t\to 0^+$, so we state it as a result for $\mu \propto e^{-f}$ itself.
In fact, the argument below is a precursor for the more delicate proof of Theorem~\ref{thm:upper_bound} and could be of interest.

\begin{prop}[Non-asymptotic defective log-Sobolev inequality]
    Assume that $f \in C^2(\R^d)$ admits a unique minimizer, that $\cpl = \cpl(f) < \infty$, and that $\Delta f \le L_0$.
    Then, $\mu$ satisfies a $(\cpl, \cpl L_0 - d)$-defective log-Sobolev inequality.
\end{prop}
\begin{proof}
    We begin with the KL divergence decomposition as in the proof of Theorem~\ref{thm:upper_bound}, except that instead of comparing to $\mc N(0,\Sigma_t)$, we compare to $\mc N(0, \cpl I)$.
    As usual, we assume without loss of generality that the minimizer is $0$ and that $f(0) = 0$.
    It yields
    \begin{align*}
        \KL{\nu}{\mu}
        &\le \frac{\cpl}{2} \int \|\nabla g\|^2\,e^{-g} + \int f e^{-g} - d + \log \frac{Z}{(2\pi \cpl)^{d/2}}\,.
    \end{align*}
    The last term is non-positive because by the quadratic growth inequality~\eqref{eqn:qg},
    \begin{align*}
        Z
        &= \int e^{-f}
        \le \int \exp\bigl(-\frac{1}{2\cpl}\,\|\cdot\|^2\bigr)
        \le (2\pi \cpl)^{d/2}\,.
    \end{align*}
    We can rewrite this as
    \begin{align*}
        \KL{\nu}{\mu}
        &\le \frac{\cpl}{2} \FI{\nu}{\mu} + \cpl \int \bigl\{\langle \nabla f,\nabla g\rangle - \frac{1}{2}\,\|\nabla f\|^2\bigr\}\,e^{-g} + \int f e^{-g} - d \\
        &\le \frac{\cpl}{2} \FI{\nu}{\mu} + \cpl \int \langle \nabla f,\nabla g\rangle \,e^{-g} - d\,,
    \end{align*}
    where we applied the PL inequality.
    Integrating by parts,
    \begin{align*}
        \int \langle \nabla f,\nabla g\rangle\,e^{-g}
        &= \int \Delta f\,e^{-g}
        \le L_0\,.
    \end{align*}
    This concludes the proof.
\end{proof}

\appendix

\section{Estimate for the normalizing constant}

In this section, we establish the following estimate for the normalizing constant.

\begin{lemma}[Normalizing constant control]\label{lem:normalizing_constant_control}
    Assume that $f \in C^2(\R^d)$ has a unique minimizer $x_\star$ and satisfies a PL inequality.
    For all $t > 0$, let $\mu_t$ be a probability measure with density $\mu_t = \frac{1}{Z_t}\,e^{-f/t}$, $f(x_\star) = 0$,
    and define $\Sigma_t \deq t\,[\nabla^2 f(x_\star)]^{-1}$.
    Then
    $$
    \log \frac{Z_t}{\sqrt{\det(2\pi \Sigma_t)}} =
    o(1)\,.
    $$
\end{lemma}
\begin{proof}
    Without loss of generality, we can assume that $x_\star = 0$.
Let $\cpl = \cpl(f)$ and $r > \sqrt{\cpl dt}$.
We split into two regions: $B_r \deq B(0,r)$ and $B_r^\comp$.
For the latter, we note that by the quadratic growth inequality~\eqref{eqn:qg}
and Gaussian concentration,
\begin{align*}
\int_{B_r^\comp} e^{-f/t} &\leqslant
\int_{B_r^\comp} \exp \Bigl(-\frac{1}{2\cpl t}\,\|\cdot\|^2\Bigr)\\
&\leqslant (2\pi \cpl t)^{d/2} \exp \Bigl(-\frac{(r - \sqrt{\cpl dt})^2}{2t\cpl}\Bigr)\,.
\end{align*}
Let us take $r = (\sqrt{d} + A)\sqrt{\cpl t}$ for some $A > 0$.
Then, the above becomes
\begin{align*}
    \int_{B_r^\comp} e^{-f/t}
    &\le (2\pi \cpl t)^{d/2} \exp\bigl(-\frac{A^2}{2}\bigr)\,.
\end{align*}

For the remaining region, let
\begin{align*}
    \rho(r) \deq \sup_{x\in B_r}{\norm{\nabla^2 f(x) - \nabla^2 f(0)}_{\rm op}}
\end{align*}
denote the modulus of continuity of the Hessian near the minimizer, and note that since $f$ is of class $C^2$, we have $\rho(r) \to 0$ as $r\to 0^+$.
Taylor expansion shows that for all $x\in B_r$,
\begin{align*}
    \Bigl\lvert f(x) -\frac{1}{2}\,\langle x, \nabla^2 f(0) \,x\rangle \Bigr\rvert 
    &= \Bigl\lvert \int_0^1 (1-t)\,\langle x, [\nabla^2 f(tx)-\nabla^2 f(0)]\,x\rangle \, \ud t\Bigr\rvert
    \le \frac{\rho(r)\,r^2}{2}\,.
\end{align*}
Hence, for $\Sigma_t =t\,[\nabla^2 f(0)]^{-1}$,
\begin{align*}
    \int_{B_r} e^{-f/t}
    &\leqslant \int_{B_r} \exp\Bigl( -\frac12\, \|x\|^2_{\Sigma_t^{-1}} + \frac{\rho(r) \,r^2}{2t}\Bigr) \\
    &\leqslant \sqrt{\det (2\pi \Sigma_t)}
    \exp\bigl(\cpl\,(d+A^2)\,\rho((\sqrt d + A)\sqrt{\cpl t})\bigr)\,.
\end{align*}
Putting these results together, we obtain
\begin{align*}
    Z_t
    &\le \sqrt{\det(2\pi \Sigma_t)}\\
    &\qquad{}\times \Bigl[\exp\bigl(\cpl\,(d+A^2)\,\rho((\sqrt d + A)\sqrt{\cpl t})\bigr) + \cpl^{d/2}\sqrt{\det(\nabla^2 f(0))} \exp\bigl(-\frac{A^2}{2}\bigr)\Bigr]\,.
\end{align*}
The result follows by letting $t\to 0^+$ first, followed by $A \to\infty$.
\end{proof}

\begin{remark}\label{rmk:quantitative_normalizing}
    To extract a quantitative estimate, let us assume that $\nabla^2 f \preceq \beta I$ and that $\nabla^3 f$ is $\gamma$-Lipschitz in the operator norm, so that $\rho(r) \le \gamma r$.
    Assume that $\cpl^{3/2}\gamma \ge 1$ for simplicity.
    If we choose $A \asymp \sqrt{d\log(\beta \cpl)} + \sqrt{\log(1/t)}$, it leads to the estimate
    \begin{align*}
        \log \frac{Z_t}{\sqrt{\det(2\pi\Sigma_t)}}
        = O\Bigl(\cpl^{3/2}\gamma\,\bigl(d \log \frac{\beta\cpl}{t}\bigr)^{3/2}\, t^{1/2}\Bigr)\,.
    \end{align*}
\end{remark}

\bibliographystyle{alpha}
\bibliography{annot}

\newcommand{\etalchar}[1]{$^{#1}$}
\begin{thebibliography}{BEGK04}

\bibitem[ACGS21]{altschuler2021averaging}
Jason Altschuler, Sinho Chewi, Patrik~R Gerber, and Austin~J. Stromme.
\newblock Averaging on the {B}ures--{W}asserstein manifold: dimension-free
  convergence of gradient descent.
\newblock {\em Advances in Neural Information Processing Systems}, 2021.

\bibitem[BBCG08]{bakry2008simple}
Dominique Bakry, Franck Barthe, Patrick Cattiaux, and Arnaud Guillin.
\newblock A simple proof of the {P}oincar{\'e} inequality for a large class of
  probability measures including the log-concave case.
\newblock {\em Electron. Commun. Probab.}, 13:60--66, 2008.

\bibitem[BEGK04]{bovier2004metastability}
Anton Bovier, Michael Eckhoff, V{\'e}ronique Gayrard, and Markus Klein.
\newblock Metastability in reversible diffusion processes. {I.} {S}harp
  asymptotics for capacities and exit times.
\newblock {\em J. Eur. Math. Soc. (JEMS)}, 6(4):399--424, 2004.

\bibitem[Ber13]{Ber13Kramer}
Nils Berglund.
\newblock Kramers' law: validity, derivations and generalisations.
\newblock {\em Markov Process. Relat. Fields}, 19(3):459--490, 2013.

\bibitem[BGL14]{bakry2014analysis}
Dominique Bakry, Ivan Gentil, and Michel Ledoux.
\newblock {\em Analysis and geometry of Markov diffusion operators}, volume
  103.
\newblock Springer, 2014.

\bibitem[CEL{\etalchar{+}}24]{Che+24LMC}
Sinho Chewi, Murat~A. Erdogdu, Mufan~(B.) Li, Ruoqi Shen, and Matthew~S. Zhang.
\newblock Analysis of {L}angevin {M}onte {C}arlo from {P}oincar\'{e} to
  log-{S}obolev.
\newblock {\em Found. Comput. Math.}, 24(4), 2024.

\bibitem[CGLL23]{Che+23FisherLower}
Sinho Chewi, Patrik~R. Gerber, Holden Lee, and Chen Lu.
\newblock Fisher information lower bounds for sampling.
\newblock In Shipra Agrawal and Francesco Orabona, editors, {\em Proceedings of
  the 34th International Conference on Algorithmic Learning Theory}, volume 201
  of {\em Proceedings of Machine Learning Research}, pages 375--410. PMLR, 2
  2023.

\bibitem[CGW10]{cattiaux2010note}
Patrick Cattiaux, Arnaud Guillin, and Li-Ming Wu.
\newblock A note on {T}alagrand’s transportation inequality and logarithmic
  {S}obolev inequality.
\newblock {\em Probab. Theory Relat. Fields}, 148:285--304, 2010.

\bibitem[Che24]{chewi2023log}
Sinho Chewi.
\newblock Log-concave sampling.
\newblock {\em Book draft available at \url{https://chewisinho.github.io}},
  2024.

\bibitem[CM14]{colding2014lojasiewicz}
Tobias~H. Colding and William~P. {Minicozzi II}.
\newblock {\L{}}ojasiewicz inequalities and applications.
\newblock {\em arXiv preprint arXiv:1402.5087}, 2014.

\bibitem[CMRS20]{chewi2020gradient}
Sinho Chewi, Tyler Maunu, Philippe Rigollet, and Austin~J Stromme.
\newblock Gradient descent algorithms for {B}ures--{W}asserstein barycenters.
\newblock In {\em Conference on Learning Theory}, pages 1276--1304. PMLR, 2020.

\bibitem[CS24]{chen2024optimization}
August~Y. Chen and Karthik Sridharan.
\newblock From optimization to sampling via {L}yapunov potentials.
\newblock {\em arXiv preprint arXiv:2410.02979}, 2024.

\bibitem[Dal17]{dalalyan2017further}
Arnak Dalalyan.
\newblock Further and stronger analogy between sampling and optimization:
  {L}angevin {M}onte {C}arlo and gradient descent.
\newblock In {\em Conference on Learning Theory}, pages 678--689. PMLR, 2017.

\bibitem[Eyr35]{eyring1935activated}
Henry Eyring.
\newblock The activated complex in chemical reactions.
\newblock {\em J. Chem. Phys.}, 3(2):107--115, 1935.

\bibitem[GBK05]{gayrard2005metastability}
V{\'e}ronique Gayrard, Anton Bovier, and Markus Klein.
\newblock Metastability in reversible diffusion processes {II}: precise
  asymptotics for small eigenvalues.
\newblock {\em J. Eur. Math. Soc. (JEMS)}, 7(1):69--99, 2005.

\bibitem[GM91]{gelfand1991recursive}
Saul~B. Gelfand and Sanjoy~K. Mitter.
\newblock Recursive stochastic algorithms for global optimization in {$\mathbb
  R^d$}.
\newblock {\em SIAM J. Control Optim.}, 29(5):999--1018, 1991.

\bibitem[HKS89]{holley1989asymptotics}
Richard~A. Holley, Shigeo Kusuoka, and Daniel~W. Stroock.
\newblock Asymptotics of the spectral gap with applications to the theory of
  simulated annealing.
\newblock {\em J. Funct. Anal}, 83(2):333--347, 1989.

\bibitem[KNS16]{karimi2016linear}
Hamed Karimi, Julie Nutini, and Mark Schmidt.
\newblock Linear convergence of gradient and proximal-gradient methods under
  the {{P}olyak--{{\L}{}}ojasiewicz condition}.
\newblock {\em Joint European Conference on Machine Learning and Knowledge
  Discovery in Databases}, 2016.

\bibitem[Kra40]{kramers1940brownian}
Hendrik~A. Kramers.
\newblock Brownian motion in a field of force and the diffusion model of
  chemical reactions.
\newblock {\em Physica}, 7(4):284--304, 1940.

\bibitem[KS22]{kinoshita2022improved}
Yuri Kinoshita and Taiji Suzuki.
\newblock Improved convergence rate of stochastic gradient {L}angevin dynamics
  with variance reduction and its application to optimization.
\newblock {\em Advances in Neural Information Processing Systems},
  35:19022--19034, 2022.

\bibitem[LE23]{li2023riemannian}
Mufan~(B.) Li and Murat~A. Erdogdu.
\newblock Riemannian {L}angevin algorithm for solving semidefinite programs.
\newblock {\em Bernoulli}, 29(4):3093--3113, 2023.

\bibitem[Li21]{Li21Blog}
Mufan~(B.) Li.
\newblock On escape time, {L}yapunov function, {P}oincar\'{e} inequality, and
  the {KLS} conjecture beyond convexity.
\newblock \url{https://mufan-li.github.io/lyapunov_escape/}, 2021.

\bibitem[\L{}o63]{lojasiewicz1963topological}
Stanislaw \L{}ojasiewicz.
\newblock A topological property of real analytic subsets (in {F}rench).
\newblock {\em Coll. du CNRS, Les {\'e}quations aux d{\'e}riv{\'e}es
  partielles}, 117(87-89):2, 1963.

\bibitem[LZB22]{liu2022loss}
Chaoyue Liu, Libin Zhu, and Mikhail Belkin.
\newblock Loss landscapes and optimization in over-parameterized non-linear
  systems and neural networks.
\newblock {\em Appl. Comput. Harmon. Anal.}, 59:85--116, 2022.

\bibitem[MS14]{menz2014poincare}
Georg Menz and Andr{\'e} Schlichting.
\newblock Poincar{\'e} and logarithmic {S}obolev inequalities by decomposition
  of the energy landscape.
\newblock {\em Ann. Probab.}, pages 1809--1884, 2014.

\bibitem[Ott01]{otto2001porousmedium}
Felix Otto.
\newblock The geometry of dissipative evolution equations: the porous medium
  equation.
\newblock {\em Commun. Partial Differ. Equ.}, 26(1-2):101--174, 2001.

\bibitem[OV00]{otto2000generalization}
Felix Otto and C{\'e}dric Villani.
\newblock Generalization of an inequality by {T}alagrand and links with the
  logarithmic {S}obolev inequality.
\newblock {\em J. Funct. Anal.}, 173(2):361--400, 2000.

\bibitem[Pol64]{polyak1964gradient}
Boris~T. Polyak.
\newblock Gradient methods for solving equations and inequalities (in
  {R}ussian).
\newblock {\em USSR Computational Mathematics and Mathematical Physics},
  4(6):17--32, 1964.

\bibitem[Rot85]{Rot1985}
Oscar~S. Rothaus.
\newblock Analytic inequalities, isoperimetric inequalities and logarithmic
  {S}obolev inequalities.
\newblock {\em J. Funct. Anal.}, 64(2):296--313, 1985.

\bibitem[RRT17]{raginsky2017non}
Maxim Raginsky, Alexander Rakhlin, and Matus Telgarsky.
\newblock Non-convex learning via stochastic gradient {L}angevin dynamics: a
  nonasymptotic analysis.
\newblock In {\em Conference on Learning Theory}, pages 1674--1703. PMLR, 2017.

\bibitem[RS22]{rigollet2022sample}
Philippe Rigollet and Austin~J. Stromme.
\newblock On the sample complexity of entropic optimal transport.
\newblock {\em arXiv preprint 2206.13472}, 2022.

\bibitem[Str23]{stromme2023minimum}
Austin~J. Stromme.
\newblock Minimum intrinsic dimension scaling for entropic optimal transport.
\newblock {\em arXiv preprint 2306.03398}, 2023.

\bibitem[TLR18]{tzen2018local}
Belinda Tzen, Tengyuan Liang, and Maxim Raginsky.
\newblock Local optimality and generalization guarantees for the {L}angevin
  algorithm via empirical metastability.
\newblock In {\em Conference On Learning Theory}, pages 857--875. PMLR, 2018.

\bibitem[Vil03]{villani2003topics}
C\'{e}dric Villani.
\newblock {\em Topics in optimal transportation}, volume~58 of {\em Graduate
  Studies in Mathematics}.
\newblock American Mathematical Society, Providence, RI, 2003.

\bibitem[Wan14]{Wang14Diffusion}
Feng-Yu Wang.
\newblock {\em Analysis for diffusion processes on {R}iemannian manifolds},
  volume~18 of {\em Advanced Series on Statistical Science \& Applied
  Probability}.
\newblock World Scientific Publishing Co. Pte. Ltd., Hackensack, NJ, 2014.

\bibitem[Wan24]{wang2024uniform}
Songbo Wang.
\newblock Uniform log-{S}obolev inequalites for mean field particles with
  flat-convex energy.
\newblock {\em arXiv preprint arXiv:2408.03283}, 2024.

\bibitem[XCZG17]{xu2017global}
Pan Xu, Jinghui Chen, Difan Zou, and Quanquan Gu.
\newblock Global convergence of {L}angevin dynamics based algorithms for
  nonconvex optimization.
\newblock {\em arXiv preprint arXiv:1707.06618}, 2017.

\bibitem[ZLC17]{zhang2017hitting}
Yuchen Zhang, Percy Liang, and Moses Charikar.
\newblock A hitting time analysis of stochastic gradient {L}angevin dynamics.
\newblock In {\em Conference on Learning Theory}, pages 1980--2022. PMLR, 2017.

\end{thebibliography}
\end{document}